\documentclass{article}
\usepackage{graphicx, tikz, color,url}
\usepackage{amsmath,amssymb,latexsym,amsthm}
\usepackage{amsfonts}
\usetikzlibrary{arrows}
\usetikzlibrary{positioning, quotes}
\usepackage{soul}

\newtheorem{theorem}{Theorem}[section]

\newtheorem{lemma}[theorem]{Lemma}
\newtheorem{proposition}[theorem]{Proposition}
\newtheorem{observation}[theorem]{Observation}

\newtheorem{corollary}[theorem]{Corollary}

\newtheorem{problem}{Problem}

\newcommand{\cp}{\,\square\,}

\newcommand{\cF}{{\cal F}}

\begin{document}
\title{On the isolation numbers in graph products}
\author{
Boštjan Brešar$^{a,b}$\\
\and
Douglas F.\ Rall$^{c  }$\\
}
\maketitle

\begin{center}
$^a$ Faculty of Natural Sciences and Mathematics, University of Maribor, Slovenia\\
$^b$ Institute of Mathematics, Physics and Mechanics, Ljubljana, Slovenia\\
$^c$ Emeritus Professor of Mathematics, Furman University, Greenville, SC, USA\\
\end{center}
\medskip

\begin{abstract}
As a continuation of a previous study of isolation numbers in Cartesian and lexicographic products, we investigate isolation numbers and, more generally, $\cF$-isolation numbers in direct, strong, lexicographic, and Cartesian products of graphs. For direct products, we derive upper bounds for the $\{K_{n_1,\ldots,n_d}\}$-isolation number in terms of isolation and total domination parameters of the factors, and establish lower bounds based on open packings. We also determine exact values for several infinite families of direct products, including $\iota(P_{4\ell}\times C_{2k+1})=\ell(k+1)$. For strong products, we prove a general  lower bound on $\iota(G\,\boxtimes\, H,{\cal F})$ involving the $2$-packing number and provide an upper bound on $\iota(G\,\boxtimes\, H)$. For lexicographic products, we determine the $\cF$-isolation number in several general settings, obtaining exact formulas in terms of domination and total domination numbers of the first factor. Finally, for Cartesian products, we extend results from our previous work to arbitrary graph families $\cF$. We introduce $\cF$-isolation graphs and use $\cF$-transversals to derive general upper bounds, together with corresponding lower bounds.
\end{abstract}

\noindent
{\bf Keywords:}  isolation number of a graph, isolating set, Cartesian product, direct product, strong product, lexicographic product.\\

\noindent
{\bf AMS Subj.\ Class.\ (2020)}: 05C69, 05C76

\maketitle

\section{Introduction}

The isolation number of a graph, introduced by Caro and Hansberg~\cite{ch-2017}, is a natural relaxation of the domination number. Instead of requiring every vertex of a graph to be dominated, one seeks a smallest set of vertices whose closed neighborhood intersects every edge of the graph. More generally, given a family of graphs $\mathcal F$, the $\mathcal F$-isolation number $\iota(G,\mathcal F)$ is the minimum cardinality of a vertex set whose closed neighborhood removal eliminates all subgraphs belonging to $\mathcal F$. This parameter unifies several domination-type concepts and has attracted considerable attention in recent years; see, for instance, \cite{BBS,borg,BLMS, bg-2024, ch-2017,CC,lms-2024, tjk-2019} and the references therein. 

Isolation parameters in graph products form a natural line of investigation, combining two classical areas of graph theory: domination-type problems and graph products. Research on domination in graph products has a long history, motivated in part by Vizing's conjecture and its numerous variants; see the survey~\cite{BDG-12} and the monographs~\cite{HIK,HHH3}. In a recent paper~\cite{iso} we studied, together with co-authors, isolation numbers in Cartesian and lexicographic products. The purpose of the present paper is to continue this investigation from a broader perspective by considering $\mathcal F$-isolation numbers in all four standard graph products. 

A recurring theme throughout the paper is the relationship between isolation parameters and structural properties of graph products. Several of our results depend on understanding when a prescribed graph occurs as a subgraph of a product. Such questions have a long tradition in the theory of graph products and are often considerably more delicate than analogous questions concerning graph invariants. A notable recent contribution in this direction is the work of Hickingbotham and Wood~\cite{HW}, who established structural characterizations for the occurrence of complete multipartite graphs in direct (and other) products. Their theorem provides a key ingredient in our analysis of $\{K_{n_1,\ldots,n_d}\}$-isolation numbers in direct products and leads to general upper bounds as well as several exact formulas. Structural subgraph questions also arise naturally in the Cartesian product. In this setting, an important role is played by S-prime graphs, introduced by Lamprey and Barnes~\cite{LB}; see also~\cite{hel-2013,HOS}. Recall that a graph $X$ is S-prime (with respect to the Cartesian product) if every embedding of $X$ into a Cartesian product $G\cp H$ forces $X$ to be contained in one of the factors. This property allows one to transfer information about forbidden subgraphs from the factors to the product. We exploit this idea to derive general upper bounds on $\mathcal F$-isolation numbers of Cartesian products when the members of $\mathcal F$ are S-prime. 

The paper is organized as follows. In Section~\ref{sec:dir} we study direct products. We establish upper bounds for the $\{K_{n_1,\ldots,n_d}\}$-isolation number based on total domination and derive complementary lower bounds involving open packings. As a consequence, several exact formulas are obtained, including the determination of the isolation number of $P_{4\ell}\times C_{2k+1}$. Section~\ref{sec:str} is devoted to strong products, where we prove general lower and upper bounds for $\mathcal F$-isolation numbers. In Section~\ref{sec:lex} we investigate lexicographic products and determine the $\mathcal F$-isolation number in several contexts according to the value of $\iota(H,\mathcal F)$.  In particular, when $\iota(H,\cF)\ge 2$, the $\cF$-isolation number of $G\circ H$ equals the total domination number of $G$. 
Finally, in Section~\ref{sec:car} we consider Cartesian products, extending several known results on isolation numbers to arbitrary graph families and obtaining bounds in terms of isolation graphs, $\mathcal F$-transversals, and S-prime graphs.

\section{Notation and preliminaries}

Throughout the paper, all graphs are finite, simple, and undirected. For a graph $G$, we denote by $V(G)$ and $E(G)$ its vertex set and edge set, respectively. The order of $G$ is denoted by $n(G)=|V(G)|$. For a vertex $v\in V(G)$, its open neighborhood is denoted by $N_G(v)$, while $N_G[v]=N_G(v)\cup\{v\}$ denotes its closed neighborhood. The degree, $\deg_G(v)$ of $v\in V(G)$ is defined as $|N_G[v]|$. For a set $S\subseteq V(G)$, we write $N_G(S)=\bigcup_{v\in S}N_G(v)$ and $N_G[S]=\bigcup_{v\in S}N_G[v]$. If no confusion can arise, the subscript $G$ is omitted. The maximum degree of $G$ is denoted by $\Delta(G)$.
For a positive integer $d$, we use $[d]=\{1,\ldots,d\}$.

A set $D\subseteq V(G)$ is a {\em dominating set} of $G$ if every vertex in $V(G) -  D$ has a neighbor in $D$. The minimum cardinality of a dominating set is the {\em domination number} of $G$, denoted by $\gamma(G)$. A set $D\subseteq V(G)$ is a {\em total dominating set} of $G$ if every vertex of $G$ has a neighbor in $D$. The minimum cardinality of a total dominating set is the {\em total domination number} of $G$, denoted by $\gamma_t(G)$.
A set $P\subseteq V(G)$ is a {\em $2$-packing} if the distance between every two distinct vertices of $P$ is at least three. The maximum cardinality of a $2$-packing in $G$ is the {\em $2$-packing number} of $G$, denoted by $\rho_2(G)$. A set $P\subseteq V(G)$ is an {\em open packing} if the open neighborhoods of its vertices are pairwise disjoint. The maximum cardinality of an open packing in $G$ is the {\em open packing number} of $G$, denoted by $\rho^o(G)$.
The {\em clique number} of a graph $G$, denoted by $\omega(G)$, is the order of a largest complete subgraph of $G$. For a positive integer $k$, the parameter $\alpha_k(G)$ denotes the maximum order of an induced $k$-colorable subgraph of $G$.

Let $\mathcal F$ be a family of graphs. A set $S\subseteq V(G)$ is an {\em $\mathcal F$-isolating set} of $G$ if the graph induced by $V(G) - N_G[S]$ contains no subgraph isomorphic to a member of $\mathcal F$. The minimum cardinality of an $\mathcal F$-isolating set of $G$ is the {\em $\mathcal F$-isolation number} of $G$ and is denoted by $\iota(G,\mathcal F)$. An $\mathcal F$-isolating set of cardinality $\iota(G,\mathcal F)$ is called an {\em $\iota(G,\mathcal F)$-set}. If $\mathcal F=\{F\}$, we write $\iota(G,F)$ instead of $\iota(G,\{F\})$. In particular, when $F=K_2$, the parameter $\iota(G,K_2)$ is the usual {\em isolation number} of $G$ and is denoted by $\iota(G)$. An $\iota(G)$-set is an isolating set of cardinality $\iota(G)$.
An {\em $\mathcal F$-transversal} of a graph $G$ is a set $T\subseteq V(G)$ that intersects the vertex set of every subgraph of $G$ isomorphic to a member of $\mathcal F$. The minimum cardinality of an $\mathcal F$-transversal of $G$ is denoted by $\beta_{\mathcal F}(G)$. If $\cF=\{K_2\}$, then an $\cF$-transversal is called a {\em vertex cover}.
A set that is both an $\mathcal F$-transversal and a dominating set of $G$ is called a {\em dominating $\mathcal F$-transversal}. The minimum cardinality of a dominating $\mathcal F$-transversal of $G$ is denoted by $\beta_{\mathcal F}^{\rm dom}(G)$.

Let $\mathcal I(G,\mathcal F)$ be the family of all $\iota(G,\mathcal F)$-sets. For $A\in\mathcal I(G,\mathcal F)$, let $L_A=V(G) -  N_G[A]$. The {\em $\mathcal F$-isolation graph} of $G$, denoted by $I_{\mathcal F}(G)$, is the graph whose vertex set is $\mathcal I(G,\mathcal F)$, in which two sets $A,B\in\mathcal I(G,\mathcal F)$ are adjacent if and only if $L_A\cap L_B=\emptyset$.

We next recall the graph products used in this paper. Let $G$ and $H$ be graphs. Each of the following products has vertex set $V(G)\times V(H)$. The {\em Cartesian product} $G\cp H$ is the graph in which distinct vertices $(g,h)$ and $(g',h')$ are adjacent if and only if either $g=g'$ and $hh'\in E(H)$, or $h=h'$ and $gg'\in E(G)$. The {\em direct product} $G\times H$ is the graph in which distinct vertices $(g,h)$ and $(g',h')$ are adjacent if and only if $gg'\in E(G)$ and $hh'\in E(H)$. The {\em strong product} $G\boxtimes H$ is the graph with $E(G\boxtimes H)=E(G \cp H) \cup E(G \times H)$.
The {\em lexicographic product} $G\circ H$ is the graph in which distinct vertices $(g,h)$ and $(g',h')$ are adjacent if and only if $gg'\in E(G)$, or $g=g'$ and $hh'\in E(H)$.

For a vertex $g\in V(G)$, the set $\{g\}\times V(H)$ is called the {\em $H$-fiber} over $g$ and is denoted by $^{g}\!H$. Analogously, for $h\in V(H)$, the set $V(G)\times\{h\}$ is called the {\em $G$-fiber} over $h$ and is denoted by $G^h$. For a set $X\subseteq V(G)\times V(H)$, the {\em projections} of $X$ onto $V(G)$ and $V(H)$ are, respectively,
\[
p_G(X)=\{g\in V(G):(g,h)\in X\text{ for some }h\in V(H)\}
\]
and
\[
p_H(X)=\{h\in V(H):(g,h)\in X\text{ for some }g\in V(G)\}.
\]
A graph $X$ is {\em S-prime with respect to the Cartesian product} if, for all graphs $G$ and $H$, the fact that $X$ is isomorphic to a subgraph of $G\cp H$ implies that $X$ is isomorphic to a subgraph of $G$ or of $H$.

\section{Direct product}
\label{sec:dir}

In the main result of this section, we will need the following result due to Hickingbotham and Wood~\cite{HW} concerning the existence of a complete multipartite graph $K_{d_1,\ldots,d_n}$ as a subgraph in a direct product of two graphs.

\begin{theorem} {\rm \cite[Theorem 6]{HW}} For all integers $d \ge 2$ and positive integers $n_1,\ldots,n_d$, and for all graphs $G_1$ and $G_2$ with nonempty edge sets, $K_{n_1,\ldots,n_d}$ is a subgraph of $G_1 \times G_2$ if and only if there exist positive integers $a_i,b_i$ $(1 \le i \le d)$ such that $ K_{a_1,\ldots,a_d}$ is a subgraph of $G_1$, $K_{b_1,\ldots,b_d}$ is a subgraph of $G_2$ and $n_i \le a_i b_i$ for every $i \in [d]$.
\label{thm:wood}
\end{theorem}
Using Theorem~\ref{thm:wood}, we now present an upper bound for the $K_{n_1,\ldots,n_d}$-isolation number in the direct product.

\begin{theorem}
\label{thm:direct}
If $G$ and $H$ are arbitrary graphs that have no isolated vertices and $n_1,\ldots,n_d$, where $d\ge 2$ are positive integers, then 
$$
\iota(G\times H, K_{n_1,\ldots,n_d})\le \min\big\{\iota(G,K_d)\gamma_t(H),\iota(H,K_d)\gamma_t(G)\big\}.
$$
\end{theorem}
\begin{proof}
Let $A$ be a $\gamma_t(G)$-set and let $A^*=V(G) - A$. Let $B$ be a $\iota(H,K_d)$-set, $B^*=N_H(B)-B$, and $R=V(H)-N_H[B]$. In addition, let $B_2$ be the set of vertices in $B$ having no neighbor in $B$, while $B_1=B-B_2$. We claim that $A\times B$ is an $K_{n_1,\ldots,n_d}$-isolating set in $G\times H$. 

First, we show that $$N_{G\times H}[A\times B]=(A\times B)\cup (A^*\times B_1)\cup (V(G)\times B^*).$$
Let $(g,h)\in A^*\times B_1$. Since $g\in A^*$, there exists a vertex $g'\in A$ such that $gg'\in E(G)$, and since $h\in B_1$, there exists a vertex $h'\in B_1$ such that $hh'\in E(H)$. Consequently, $(g',h')\in A\times B$ dominates $(g,h)$. Now, let $(g,h)\in V(G)\times B^*$. Since $A$ is a total dominating set of $G$, there exists a neighbor $g'\in A$ of the vertex $g$. In addition, $h$ has a neighbor $h'\in B$. Thus, $(g,h)$ is dominated by $(g',h')\in A\times B$. 

It remains to consider the subgraph of $G\times H$ induced by $(A^*\times B_2)\cup (V(G)\times R)$. Since $B_2$ is an independent set, it follows that $A^*\times B_2$ is an independent set in $G\times H$. In addition, there are no edges between $A^*\times B_2$ and $V(G)\times R$ since there are no edges between $B_2$ and $R$ in $H$. Therefore, it remains to check that the subgraph induced by $V(G)\times R$ does not contain $K_{n_1,\dots,n_d}$ as a subgraph. Since $B$ is an $\iota(H,K_d)$-isolating set, the graph $H[R]$ does not contain $K_d$, which is the complete $d$-partite graph $K_{1,\ldots,1}$. Now, by Theorem~\ref{thm:wood} we derive that $G\times H[R]$ does contain $K_{n_1,\ldots,n_d}$ as a subgraph. Indeed, since  $K_{1,\ldots,1}$ is not a subgraph of $H[R]$, it is clear that $H[R]$ also does not admit a subgraph $K_{b_1,\ldots,b_d}$ for any positive integers $b_1,\dots, b_d$. All in all, $A\times B$ is an $K_{n_1,\ldots,n_d}$-isolating set in $G\times H$ and $|A\times B|=\gamma_t(G)\iota(H,K_d)$. Reversing the roles of $G$ and $H$, the stated inequality follows. 
\end{proof}

A special case of Theorem~\ref{thm:direct} when $n_i=1$ for all $i\in [d]$, leads to the following corollary.

\begin{corollary}
 If $G$ and $H$ are arbitrary graphs that have no isolated vertices, then 
$$
\iota(G\times H, K_d)\le \min\big\{\iota(G,K_d)\gamma_t(H),\iota(H,K_d)\gamma_t(G)\big\}.
$$   
\end{corollary}

We remark that the key fact that enables the above formula is that $G\times H$ contains $K_d$ if and only if each of the factors contains $K_d$. The equivalent statement in which $K_d$ is replaced by some other graph $F$ will, in general, not hold. For instance, $\iota(K_3\times K_3,C_5)=1>0=\iota(K_3,C_5)$. 

The above corollary immediately leads to a bound on $\iota(G\times H)$ by letting $d=2$.
\begin{corollary}
\label{cor:iotadirect}
    If $G$ and $H$ are arbitrary graphs that have no isolated vertices, then  
   $$ \iota(G\times H)\le \min\big\{\iota(G)\gamma_t(H),\iota(H)\gamma_t(G)\big\}.$$
\end{corollary}
 We will show that the bound in Corollary~\ref{cor:iotadirect} is widely sharp.

We follow with a lower bound on $\iota(G\times H)$, where $G$ is restricted to graphs that admit a maximum open packing that induces a matching and $H$ is arbitrary. For a large family of graphs satisfying the former condition, recall from~\cite[Theorem 2.2]{bkr-2024} that $\rho^o(X\Box K_2)=2\rho_2(X)$ holds for any bipartite graph $X$, and one can easily derive that $X\Box K_2$ has a maximum open packing inducing a matching. 

\begin{theorem}
\label{thm:directregular}
If $G$ is a graph that has a maximum open packing which induces a matching and $H$ is any graph that has at least one edge, then
$$\iota\bigl(G\times H\bigr)\ge \frac{\rho^o(G)}{2}\cdot\bigg\lceil\frac{2|E(H)|}{\Delta(H)^2}\bigg\rceil.$$
\end{theorem}
\begin{proof}
Let $G$ be a graph whose maximum open packing $P$ induces a matching. Let $P=\{u_1,v_1,\ldots,u_k,v_k\}$ such that $u_iv_i\in E(G)$ for all $i\in [k]$ and note that $\rho^{o}(G)=2k$. Let $\Delta=\Delta(H)$, 
and $S$ an arbitrary isolating set of $G\times H$.

For each $i\in[k]$, let $$B_i=(N_G(u_i)\cup N_G(v_i))\times V(H).$$
and set $S_i=B_i\cap S$.

Note that $N[S]$ is a vertex cover of $G\times H$, hence every edge in $G\times H$ is incident with a vertex in $N[S]$. In particular, every edge between the fibers  $^{u_i}\!H$ and $^{v_i}\!H$ is incident with a vertex from $N[S]$, more precisely with a vertex from $N[S_i]$. Let $M_i$ be the set of edges between the fibers  $^{u_i}\!H$ and $^{v_i}\!H$.
For a vertex $(x,y)\in S_i$ let us count the number of edges in $M_i$ that are incident with a vertex in $N[(x,y)]$. If $x\notin \{u_i,v_i\}$, then we may assume without loss of generality, that $x\in N_G(u_i)$. Note that 
$N[(x,y)]$ contains at most $\Delta$ vertices in $^{u_i}\!H$, which are altogether  incident with at most $\Delta^2$ edges in $M_i$. On the other hand, if $x\in \{u_i,v_i\}$,  say $x=u_i$, then $(x,y)$ itself is incident with at most $\Delta$ edges in $M_i$, while every $(v_i,h)\in N(x,y)\cap\, ^{v_i}\!H$ is incident with at most $\Delta-1$ additional edges in $M_i$ (that are not incident with $(x,y)$). Altogether the vertices in $N[(x,y)]$ are incident with at most $\Delta+\Delta(\Delta-1)$ edges in $M_i$. Hence, in either case for a vertex $(x,y)\in S_i$ there are at most $\Delta^2$ edges in $M_i$ incident with a vertex in $N[(x,y)]$. Therefore, $$|S_i|\Delta^2\ge |M_i|=\sum_{h\in V(H)}{\deg_H(h)}=2|E(H)|, $$
which implies $|S_i|\ge \frac{2|E(H)|}{\Delta^2}$.  Since $|S_i|$ is an integer, we get  $|S_i|\ge \Big\lceil\frac{2|E(H)}{\Delta^2}\Big\rceil.$
Therefore, $$\iota(G\times H)=|S|=\sum_{i=1}^{k}{|S_i|}\ge \frac{\rho^{o}(G)}{2}\bigg\lceil\frac{2|E(H)|}{\Delta^2}\bigg\rceil,$$
as claimed.
\end{proof}

From the above theorem, we immediately get the following corollary for the case when $H$ is a regular graph.
\begin{corollary}
\label{cor:directregular}
If $G$ is a graph whose maximum open packing induces a matching and $H$ is an $r$-regular graph for a positive integer $r$, then
$$\iota\bigl(G\times H\bigr)\ge \frac{\rho^o(G)}{2} \bigg\lceil\frac{n(H)}{r}\bigg\rceil.$$
\end{corollary}

Combining Corollaries~\ref{cor:iotadirect} and~\ref{cor:directregular} we obtain the following formula for the isolation number of direct products of paths divisible by $4$ and odd cycles.

\begin{corollary} \label{cor:sharp}
If $k$ and $\ell$ are positive integers, then
$$\iota\bigl(P_{4\ell}\times C_{2k+1}\bigr)=\ell(k+1).
$$
\end{corollary}
\begin{proof}
By Corollary~\ref{cor:iotadirect}, we have $$\iota(P_{4\ell}\times C_{2k+1})\le \min\big\{\iota(P_{4\ell})\gamma_t(C_{2k+1}),\iota(C_{2k+1})\gamma_t(P_{4\ell})\big\}=\ell(k+1).$$
By Corollary~\ref{cor:directregular},
$$\iota(P_{4\ell}\times C_{2k+1})\ge \ell\bigg\lceil\frac{2k+1}{2}\bigg\rceil=\ell(k+1),$$
which proves the result.
\end{proof}

Corollary~\ref{cor:sharp} shows that the upper bound in Corollary~\ref{cor:iotadirect} is sharp.  Even more is true as the following observation shows.
 
\begin{observation}
 For any $x\in\mathbb{N}$ and any integer $y\ge 2$ there exist graphs $G$ and $H$ such that $\iota(G)=x, \gamma_t(H)=y$ and $\iota(G\times H)=xy.$
\end{observation}
Note that equality in the above observation is obtained by letting $G=P_{4x}$ and $H=C_{2y-1}$.  

\section{Strong product}
\label{sec:str} 

We start this section with a lower bound on the isolation number of the strong product of two graphs. Compare it with~\cite[Theorem 3.6]{iso}, where the same lower bound was proved for the Cartesian product of two graphs. The proof also follows similar lines.  

Let $X$ and $Y$ be graphs, where $V(X)=\{x_1,\ldots,x_n\}$. The {\em corona} $X\odot Y$ is obtained from a copy of $X$ and $n$ copies of $Y$ by connecting all vertices of the $i^{\rm th}$ copy of $Y$ with vertex $x_i$.  
\begin{theorem}
\label{thm:strongupper}
If $\cF$ is a family of graphs, and $G$ and $H$ are arbitrary graphs, then
$$\iota(G\,\boxtimes\, H,\cF)\ge \max\{\rho_2(G)\iota(H,\cF),\rho_2(H)\iota(G,\cF)\},$$
and the bound is sharp for any family $\cF$.
\end{theorem}
\begin{proof}
By symmetry, it suffices to prove that $\iota(G\,\boxtimes\, H,\cF)\ge \rho_2(G)\iota(H,\cF)$.  Suppose to the contrary that $\iota(G\,\boxtimes\, H,\cF)< \rho_2(G)\iota(H,\cF)$. Let $D$ be an $\iota(G\,\boxtimes\, H,\cF)$-set and let $A$ be a $\rho_2(G)$-set. Thus, there exists a vertex $x\in A$ such that $$|(N_G[x]\times V(H))\cap D|<\iota(H,\cF).$$
Let $D'=p_H\big((N_G[x]\times V(H))\cap D\big)$, and note that $|D'|<\iota(H,\cF)$. Therefore, there exists a set of vertices $L$ in $V(H)-N_H[D']$ such that $H[L]\in \cF$. We in turn infer that $\{x\}\times L\subset V(G\boxtimes H)-N[D]$ and $\{x\}\times L$ induces a graph in $\cF$, which is a contradiction to $D$ being an $\cF$-isolating set in $G\boxtimes H$.

For the sharpness of the bound, let $F\in \cF$ be an arbitrarily chosen graph from the family. Let $G=H=X\odot F$, where $X$ is an arbitrary graph. It can be verified that $\iota(G,\cF)=\gamma(G)=\rho_2(G)=n(X)$. Letting $A$ be a $\gamma(G)$-set, it is clear that $A\times A$ is a dominating set of $G\boxtimes H$. Hence, $$\rho_2(G)\iota(H,\cF)\le \iota(G\boxtimes H,\cF)\le \gamma(G)\gamma(H)=\rho_2(G)\iota(H,\cF),$$
where the first inequality follows from the statement of the theorem.
\end{proof}

Next, we present an upper bound on $\iota(G\boxtimes H)$. 
Given a graph $G$ let ${\cal I}(G)=\{S:\, S \textrm{ is an }\iota(G)\textrm{-set}\}.$ Let $$r_{\iota}(G)=\min\{|R|:\,R=V(G)-N[S],\textrm{ where }S\in {\cal I}(G)\}.$$
That is, $r_{\iota}(G)$ is the minimum number of vertices that remain undominated by an $\iota(G)$-set.
\begin{proposition}
\label{prp:stronglower}  
If $G$ and $H$ are arbitrary graphs, then $$\iota(G\boxtimes H)\le \iota(G)\iota(H)+r_{\iota}(G)\iota(H)+r_{\iota}(H)\iota(G),$$
and the bound is sharp.
\end{proposition}
\begin{proof}
Let $A$ be an $\iota(G)$-set with $R_A=V(G)-N_G[A]$ and $|R_A|=r_{\iota}(G)$, and let $B$ be an $\iota(H)$-set with $R_B=V(H)-N_G[B]$ and $|R_B|=r_{\iota}(H)$. Let $$S=(A\times B)\cup (A\times R_B)\cup (R_A\times B).$$
Note that $V(G\boxtimes H)-N[S]=R_A\times R_B$, and $R_A\times R_B$ is an independent set in $G\boxtimes H$. Therefore $S$ is an isolating set, and $|S|=\iota(G)\iota(H)+r_{\iota}(G)\iota(H)+r_{\iota}(H)\iota(G).$

For the sharpness of the bound, one can take a family from the proof of Theorem~\ref{thm:strongupper}. More precisely, letting $G=H=X\odot K_2$, where $X$ is an arbitrary graph, we get $\iota(G)=\gamma(G)=\rho_2(G)=n(X)$. In addition, $r_{\iota}(G)=0$, thus the upper bound simplifies to $\iota(G\boxtimes H)\le \iota(G)\iota(H)$.  This inequality holds as equality since $$\iota(G\boxtimes H)\ge \rho_2(G)\iota(H)=\iota(G)\iota(H),$$
where the first inequality follows from  Theorem~\ref{thm:strongupper}.
\end{proof}

\section{Lexicographic product}
\label{sec:lex}

In this section we consider the $\cF$-isolation number of lexicographic products. When $\iota(H,\cF)\ge 1$, we extend the results for $\iota(G\circ H)$ from~\cite{iso} to $\iota(G\circ H,\cF)$ for an arbitrary family $\cF$. The case when $\iota(H,\cF)=0$ is more challenging, since the structural properties of the graphs from family $\cF$ within $G\circ H$ play an important role in determining  $\iota(G\circ H,\cF)$. For this reason, when $\iota(H,\cF)=0$    we focus on the case when ${\cal F}=\{K_n\}$.

Recall that $\omega(G\circ H)=\omega(G)\omega(H)$, and there exists a maximum  complete subgraph in $G\circ H$ induced by $A\times B$, where $A$ and $B$ induce maximum complete subgraphs in $G$ and $H$, respectively.

\begin{lemma}
If $G$ and $H$ are nontrivial graphs such that $k=\omega(H)$, then $G\circ H$ has a complete subgraph of order $n$ if and only if $\omega(G)\ge \frac{n}{k}$.
\end{lemma}

\begin{proposition}
\label{prp:lex}
   Let $G$ and $H$ be nontrivial graphs with $k=\omega(H)$ and $n\ge 2$.  If $\iota(H,K_n)=0$, then $$\iota(G \circ H,K_n)=\iota(G,K_r),$$ where $r=\lceil\frac{n}{k}\rceil$.
\end{proposition}
\begin{proof}
Let $S$ be an $\iota(G,K_r)$-set. Take an arbitrary $h\in V(H)$, and let $D=S\times \{h\}$. Note that $(G\circ H)-N[D]$ is the disjoint union of $(G-N_G[S])\circ H$ and at most $|S|$ copies of $H-N_H[h]$. Note that 
$$
\begin{array}{rcl}
\omega((G\circ H)-N[D])&=&\max\{\omega((G-N_G[S])\circ H),\omega(H-N_H[h])\}\\
&\le & \max\{\omega(G-N_G[S])\cdot\omega(H),k\}\\
&\le &\max\{(r-1)k,k\}.
\end{array}
$$ 
Since $r=\lceil\frac{n}{k}\rceil$, we have $(r-1)k<n$. Hence $(G\circ H)-N[D]$ contains no complete subgraph of order $n$, and so $D$ is an $\{K_n\}$-isolating set of $G\circ H$. As $|D|=|S|=\iota(G,K_r)$, we get $\iota(G \circ H,K_n)\le \iota(G,K_r)$.

Let $D$ be an $\iota(G \circ H,K_n)$-set and suppose that $|D|<\iota(G,K_r)$. Let $S=p_G(D)$, and note that $(G-N_G[S])\circ H$ is a subgraph of $(G\circ H)-N[D]$. Thus, $$\omega((G\circ H)-N[D])\ge \omega(G-N_G[S])\omega(H).$$
Since $|S|\le |D|<\iota(G,\{K_r\})$, we have $\omega(G-N_G[S])\ge r$. Therefore,  $(G\circ H)-N[D]$ contains a complete subgraph of order $k\cdot r$, which is greater than or equal to $n$, a contradiction. We conclude that $\iota(G \circ H,K_n)\ge\iota(G,K_r)$, and the proof is complete.
\end{proof}

Next, when $\iota(H,\cF)=1$, we obtain the following formula for the $\cF$-isolation number of the lexicographic product $G\circ H$, which holds also if $G$ is disconnected, and holds for an arbitrary family $\cF$.

\begin{proposition} \label{prop:lexico-small}
Let $G$ and $H$ be nontrivial graphs and let $\cF$ be an arbitrary family of graphs.  If $\iota(H,\cF)=1$, then $\iota(G \circ H,\cF)=\gamma(G)$.
\end{proposition}
\begin{proof} 
Let $\{h\}$ be an $\cF$-isolating set of $H$.  If  $D$ is a $\gamma(G)$-set, then $D \times \{h\}$ is an $\cF$-isolating set of $G \circ H$, which implies that $\iota(G \circ H,\cF) \leq \gamma(G)$.  To see the reverse inequality, let $C \subseteq V(G \circ H)$ such that $|C| < \gamma(G)$.  There exists a vertex $z$ in $G -N[p_G(C)]$, and so $^z\!H\cap N[C]=\emptyset$. Since $^z\!H$ contains a copy of at least one $F\in \cF$, we are in a contradiction with $C$ being an $\cF$-isolating set.  Therefore, $\iota(G \circ H,\cF) \geq \gamma(G)$.
\end{proof}

Finally, we consider the case when the $\cF$-isolation number of the second factor is at least $2$. The following result will be used in the proof. 

\begin{theorem}\hskip-0.5pt {\rm \cite{spt}} 
\label{thm:gammalex}
If $G$ is a nontrivial connected graph and $H$ is a connected graph with $\gamma(H)\ge 2$, then $\gamma(G \circ H)=\gamma_t(G)$.
\end{theorem}

In the next theorem, we restrict $G$ and $H$ to be connected, while for arbitrary $G$ and $H$ (with $G$ having no isolated vertices) the result can be easily extended. 

\begin{theorem} \label{thm:lexico}
If $G$ and $H$ are nontrivial connected graphs and $\cF$ is an arbitrary family of graphs such that $\iota(H,\cF)\geq 2$, then 
\[\iota(G \circ H,\cF)=\gamma(G \circ H)=\gamma_t(G)\,.\]
\end{theorem}

\begin{proof}
By definition, $\iota(G \circ H,\cF)\leq \gamma(G \circ H)$ and by Theorem~\ref{thm:gammalex}, $\gamma(G \circ H)=\gamma_t(G)$.  Thus, it remains to show that 
$\gamma_t(G) \leq \iota(G \circ H,\cF)$.

Let $S$ be a $\iota(G\circ H,\cF)$-set, which at the same time minimizes the number of $H$-fibers that contain more than one vertex from an $\iota(G\circ H,\cF)$-set. 
We claim that the projection $p_G(S)$ is a total dominating set of $G$. 
Suppose first that there is a vertex $x\in V(G)-p_G(S)$ that has no neighbor in $p_G(S)$. Then, in $G\circ H$ no vertex of the fiber $^{x}\!H$ is dominated by $S$. Since $\iota(H,\cF)\geq 2$, we derive that there exists a copy of some graph $F\in \cF$ that is a subgraph of $^{x}\!H$, a contradiction with $S$ being an $\cF$-isolating set in $G\circ H$. Thus, $p_G(S)$ is a dominating set of $G$. Suppose next that there exists a vertex $y\in p_G(S)$ such that $N_G(y)\cap p_G(S)=\emptyset$. We claim that $|S\,\cap \,^{y}\!H|=1$.

Let $|S\,\cap \,^{y}\!H|=k$. Let $(y,h_1),\ldots,(y,h_k)$ be the vertices in $S\,\cap \,^{y}\!H$. First, suppose $k>2$. Let $T=(S -  \{(y,h_2),\ldots,(y,h_k)\})\cup \{(z,h_1)\}$, where $z\in N_G(y)$. Note that $N[S]\subseteq N[T]$, which implies that $T$ is an $\cF$-isolating set of $G\circ H$. This is a contradiction, since $|T|<|S|$. Finally, suppose that $k=2$. Let $T=(S - \{(y,h_2)\})\cup \{(z,h_1)\}$, where $z\in N_G(y)$. Note that $N[S]\subseteq N[T]$, which implies that $T$ also is an $\cF$-isolating set of $G\circ H$. In addition $|T|=|S|$, yet $T$ has fewer $H$-fibers containing more than one vertex from $T$ than $S$ does, which contradicts the choice of $S$.  Therefore, $p_G(S)$ is a total dominating set of $G$, which implies that 
$\gamma_t(G) \leq |p_G(S)| \leq |S|=\iota(G \circ H,\cF)$.
\end{proof}

Note that Theorem~\ref{thm:lexico} is a direct generalization of~\cite[Theorem 4.3]{iso}. 

\section{Cartesian product}
\label{sec:car}

We start by extending an upper bound on $\iota(G\Box H)$ from~\cite{iso} to an upper bound on $\iota(G\cp H,\cF)$ where $\cF$ is an arbitrary family of graphs.

\begin{theorem}\label{thm:uperBoundProduct}
If $G$ and $H$ are graphs and $\cF$ is an arbitrary family of graphs, then
$$\iota(G\,\Box\, H,\cF)\le \min\{\alpha_k(G)\iota(H,\cF)+(n(G)-\alpha_k(G))\gamma(H),\alpha_\ell(H)\iota(G,\cF)+(n(H)-\alpha_\ell(H))\gamma(G)\},$$
where $k=\omega(I_{\cF}(H))$ and $\ell=\omega(I_{\cF}(G))$. 
\end{theorem}
The proof of the above theorem follows similar lines as the proof of~\cite[Theorem 3.2]{iso}, and we omit it. Next, we prove that the bound in Theorem~\ref{thm:uperBoundProduct} is widely sharp by using the following formula. 

\begin{proposition}
\label{prp:Cart1}
Let $\cF$ be any finite family of graphs and let $r$ be the largest order of the graphs in $\cF$. If $m$ and $n$ are positive integers such that $m\ge 3$ and $n\ge r+3$, then $$\iota(C_m\cp K_n,\cF)=m.$$
\end{proposition}
\begin{proof}
 Since $n$ is greater than the order of any graph in $\cF$, we note that $I_{\cF}(K_n)=K_n$, and so $\omega(I_{\cF}(K_n))=n$. Clearly, $\alpha_n(C_m)=m$, because $C_m$ is $3$-colorable. Thus, by Theorem~\ref{thm:uperBoundProduct}, we get $$\iota(C_m\cp K_n,\cF)\le \alpha_n(C_m)\iota(K_n,\cF)+(n(C_m)-\alpha_n(C_m))\gamma(K_n)=m,$$
and so it remains to prove the reverse inequality. 

 Let $V(C_m)=[m]$ and let $S$ be an $\iota(C_m\cp K_n,\cF)$-set. If there exists a $K_n$-fiber such that $^{i}\!K_n\cap S=\emptyset$, then (taking $i$ modulo $m$) $$|^{i-1}\!K_n\cap S|+|^{i+1}\!K_n\cap S|\ge n-r+1\ge 4\,,$$
for otherwise the subgraph of $^{i}\!K_n$ consisting of the vertices that are not dominated by $S$ contains $K_r$ and thus contains a copy of a graph $F\in\cF$ of order $r$. Consider the cyclic order of the integers $f_i=|^{i}\!K_n\cap S|$ for $i\in [m]$, and note that $\sum_{i=1}^m{f_i}=|S|$. By using the condition $$f_{i-1}+f_{i+1}\ge 4$$
whenever $f_i=0$, one can prove by induction that $\sum_{i=1}^m{f_i}\ge m$ for any cyclic order admitting that condition. Hence, $\iota(C_m\cp K_n,\cF)=|S|\ge m$.
\end{proof}

By Theorem~\ref{thm:uperBoundProduct} and Proposition~\ref{prp:Cart1}, we get $$m=\iota(C_m\cp K_n,\cF)\le \alpha_n(C_m)\iota(K_n,\cF)+(n(C_m)-\alpha_n(C_m))\gamma(K_n)=m,$$
implying that the bound in Theorem~\ref{thm:uperBoundProduct} is sharp.

Next we present an alternative upper bound on $\iota(G\circ H,\cF)$, which  generalizes~\cite[Theorem 3.5]{iso}. 
We  need to recall the following definition from~\cite{LB}.
A graph $X$ is \emph{S-prime (with respect to the Cartesian product)} if for all graphs $G$ and $H$, $X$ is a subgraph of $G\cp H$ only if $X$ is a subgraph of $G$ or of $H$. Equivalently, $X$ cannot be embedded as a subgraph of a nontrivial Cartesian product without already being contained in one of the factors. The family of S-prime graphs is quite rich; see~\cite{LB}. In particular, all complete multipartite graphs with the exception of stars and $K_{2,2}$ are S-prime; see~\cite{HW}.

\begin{theorem}
If $G$ and $H$ are graphs and $\cF$ is a family of S-prime graphs, then \[\iota(G\cp H,\cF)\le \beta_\cF^{\rm dom}(G)\beta_\cF^{\rm dom}(H)\,.\]
\end{theorem}
\begin{proof}
Let $A$ and $B$ be dominating $\cF$-transversals of  $G$ and $H$, respectively, with $|A|=\beta_\cF^{\rm dom}(G)$ and $|B|=\beta_\cF^{\rm dom}(H)$. Set $D=A\times B$, and note that $(G\cp H)-N[D]$ is isomorphic to $(G-A)\cp (H-B)$. Since neither of the factors of this product contains a graph in $\cF$, also $(G-A)\cp (H-B)$ contains no graph in $\cF$, since all graphs in $\cF$ are S-prime.  Therefore, $D$ is an $\cF$-isolating set in $G \cp H$.
\end{proof}

It is generally harder to obtain some good lower bounds on the isolation number of a graph. The following result generalizes a lower bound on $\iota(G\cp H)$ from~\cite{iso} to a lower bound on $\iota(G\cp H,\cF)$ for an arbitrary family $\cF$. Since the proof is similar to the one of~\cite[Theorem 3.6]{iso} as well as to the proof of Theorem~\ref{thm:strongupper}, we omit it. 

\begin{proposition}
If $G$ and $H$ are graphs and $\cF$ a family of graphs, then
$$\iota(G\,\Box\, H,\cF)\ge \max\{\rho_2(G)\iota(H,\cF),\rho_2(H)\iota(G,\cF)\}.$$
\end{proposition}

\section{Concluding remarks}

A recurring theme of this work is the connection between $\mathcal F$-isolation numbers and structural questions concerning the occurrence of prescribed subgraphs in graph products. Indeed, several of our results rely on understanding when a graph from a given family can appear as a subgraph of a product and, conversely, when such an occurrence is impossible. From this perspective, structural characterizations of subgraph containment and the notion of $S$-primeness are important tools in the study of $\mathcal F$-isolation numbers in graph products. 

In Section~\ref{sec:dir}, we proved an upper bound on the $K_{n_1,\ldots,n_d}$-isolation number of direct products by using the recent characterization of complete multipartite subgraphs in direct products of graphs due to Hickingbotham and Wood. Similarly, in Section~\ref{sec:lex}, Proposition~\ref{prp:lex} relies on a simple structural property describing the existence of complete subgraphs in lexicographic products. During the course of our investigations, we also obtained a characterization of complete multipartite subgraphs in lexicographic products, leading to corresponding bounds on $\iota(G\circ H,K_{n_1,\ldots,n_d})$. Since both the characterization and the resulting bounds are rather technical, we have chosen not to include them here.

Concerning the isolation number in strong products of graphs (Section~\ref{sec:str}), we found an upper and a lower bound, both of which were proved to be sharp by using corona graphs as examples of sharpness. However, these graphs are specific in the sense that their corresponding isolation numbers coincide with their domination numbers as well as $2$-packing numbers. Therefore, we pose the following problem. 
\begin{problem}
Are there graphs $G$ and $H$ and a family $\cF$ such that the inequality in Theorem~\ref{thm:strongupper}, respectively, Proposition~\ref{prp:stronglower}, is attained, yet $\iota(G,\cF)<\gamma(G)$?
\end{problem}

In Section~\ref{sec:car}, we exploit another structural aspect of graph products, namely $S$-primeness, to obtain a general bound on $\mathcal F$-isolation numbers of Cartesian products.  In strong products of graphs, S-prime graphs are well understood, since any connected graph on $n$ vertices can be isometrically embedded into a strong product of $n$ paths~\cite[Theorem 15.1]{HIK}, and thus is a subgraph of a strong product of paths. It would also be interesting to investigate S-prime graphs in other graph products, especially in the direct product whose S-prime graphs have not yet been explored.

Determining the domination number in hypercubes $Q_n$ (Cartesian powers of $K_2$) is one of the most challenging problems in domination theory; see~\cite{bkr-2024}. The same holds for the isolation number, since it was proved in~\cite{iso} that $\iota(Q_{n+1})=\gamma(Q_n)$ holds for any positive integer $n$. We can partially extend this result as follows. If $k$ and $n$ are positive integers, where $k<n$, then
\begin{equation}
\label{eq:cubes}
    \iota(Q_n,Q_k)\ge \gamma(Q_{n-k}).
\end{equation}
Indeed, suppose that  $\iota(Q_n,Q_k)<\gamma(Q_{n-k})$, and consider $Q_n$ represented as $Q_{n-k}\cp Q_k$. Letting $D$ be an $\iota(Q_{n},Q_k)$-set, then $p_{Q_{n-k}}(D)$ is not a dominating set of $Q_{n-k}$. Hence, there exists a vertex $x\in V(Q_{n-k})$ such that $^{x}\!Q_k\subset V(G\cp H)-N[D]$, which yields that there is a subgraph in $Q_n$ isomorphic to $Q_k$, which is not dominated by $N[D]$, a contradiction. Thus~\eqref{eq:cubes} holds. We have not managed to prove the reversed inequality, and leave it as an open problem.
\begin{problem}
Is it true that for any positive integers $n$ and $k$, where $k<n$, we have $$\iota(Q_n,Q_k)= \gamma(Q_{n-k})?$$ 
\end{problem}
More generally, it would be interesting to investigate $\iota(Q_n,\cF)$, where $\cF$ is a family of graphs that appear as a subgraph of a hypercube.

\section*{Acknowledgements}
During final stages of preparation of this manuscript, MS Copilot Premium was used for editing some of its parts, which were then verified by the authors who take full responsibility for the final version. 

This research was supported by an AMS-Simons Research Enhancement Grant for Primarily Undergraduate Institution Faculty.
B.B. also acknowledges the financial support of the Slovenian Research and Innovation Agency (research core funding No.\ P1-0297, projects N1-0285, N1-0431, J1-70045).


\begin{thebibliography}{4}


\bibitem{BBS} K.~Bartolo, P.~Borg and D.~Scicluna, Solution to a problem on isolation of 3-vertex paths, Discrete Math.\ 349 (2026) Paper No. 115312, 8 pp.

\bibitem{borg} P.~Borg, Isolation of regular graphs, stars and $k$-chromatic graphs, Discrete Math.\ (2026) Paper No. 114706, 11 pp.

\bibitem{BLMS} P.~Borg, M.~Lema\'{n}ska, M.~Mora, M.J.~Souto-Salorio,
Upper bounds on the $k$-isolation number, Discrete Math.\ 349 (2026), no.~10, Paper No. 115217, 17 pp.


\bibitem{bg-2024} G.~Boyer, W.~Goddard, 
Disjoint isolating sets and graphs with maximum isolation number, Discrete Appl.\ Math.\ 356 (2024) 110--116.




\bibitem{BDG-12} B.~Bre\v{s}ar, P.~Dorbec, W.~Goddard, B.L.~Hartnell, M.A.~Henning, S.~Klav\v{z}ar, D.F.~Rall, Vizing's conjecture: a survey and recent results, J. Graph Theory 69 (2012) 46--76.


\bibitem{iso} B.~Bre\v{s}ar, T.~Dravec, D.~Johnston, K.~Kuenzel, D.F.~Rall, A.~Tepeh, Isolation number: Cartesian and lexicographic
products and generalized Sierpi\'{n}ski graphs, arXiv:2508.16338; Aug 2025.

\bibitem{bkr-2024} B.~Bre\v{s}ar, S.~Klav\v{z}ar, D.F.~Rall, 
Packings in bipartite prisms and hypercubes, 
Discrete Math.\ 347 (2024) Paper No.\ 113875, 6 pp.


\bibitem{ch-2017} Y.~Caro, A.~Hansberg, Partial domination---the isolation number of a graph,   
Filomat 31 (2017) 3925--3944.

\bibitem{CC} S.~Chen and Q.~Cui, Proof of a conjecture on isolation of cycles in graphs, Discrete Math.\ 349 (2026)  Paper No. 115076.


\bibitem{HIK} R.~Hammack, W.~Imrich, S.~Klav\v{z}ar, Handbook of product graphs, Second Edition, CRC Press, Boca Raton, FL, 2011.

\bibitem{HHH3}
T.W.~Haynes, S.T.~Hedetniemi, M.A.~Henning, Domination in Graphs: Core
Concepts. Series: Springer Monographs in Mathematics, Springer, Cham, 2022.

\bibitem{hel-2013}
M.~Hellmuth, On the complexity of recognizing S-composite and S-prime graphs, Discrete Appl.\ Math.\ 161 (2013) 1006--1013.

\bibitem{HOS} M.~Hellmuth, L.~Ostermeier and P.F.~Stadler, Diagonalized Cartesian products of $S$-prime graphs are $S$-prime, Discrete Math.\  312 (2012) 74--80.

\bibitem{HW}
R.~Hickingbotham, D.R.~Wood, Structural properties of graph products, J. Graph Theory 109 (2025) 107--136.


\bibitem{spt} T.~Kraner  \v{S}umenjak, P.~Pavli\v{c},  A.~Tepeh, 
On the Roman domination in the lexicographic product of graphs.
 Discrete Appl. Math.  160  (2012) 2030--2036.

\bibitem{LB} R.H.~Lamprey, B.H.~Barnes, A new concept of primeness in graphs, Networks 11 (1981) 279--284.
 
\bibitem{lms-2024} M.~Lema\'{n}ska, M.~Mora, M.J.~Souto-Salorio, 
Graphs with isolation number equal to one third of the order,
Discrete Math.\ 347 (2024) Paper No.\ 113903, 10 pp.


\bibitem{tjk-2019} S.~Tokunaga, T.~Jiarasuksakun, P.~Kaemawichanurat, 
Isolation number of maximal outerplanar graphs,
Discrete Appl.\ Math.\ 267 (2019) 215--218.

 
\end{thebibliography}
\end{document}